\documentclass[openright,a4paper,american,11pt]{amsart}
\usepackage[latin1]{inputenc}
\usepackage{graphicx}
\usepackage{amsmath}
\usepackage{amssymb}
\usepackage[english]{babel}
\usepackage{caption}
\usepackage{subcaption}
\usepackage{amsfonts}
\usepackage{amscd}
\usepackage{wrapfig}

\usepackage[linktocpage=true,pagebackref, linktoc=all,hidelinks,hypertexnames=true,pdfstartpage=1, pdftitle="Tropical approach to Nagata's conjecture in positive characteristic", pdfauthor="Nikita
  Kalinin", pdfstartview={FitH}, pdfkeywords={tropical geometry, singularity theory}]{hyperref}

\usepackage{boxedminipage}
\usepackage{hyperref}
\usepackage{tikz}

\renewcommand*{\backref}[1]{}
\renewcommand*{\backrefalt}[4]{%
  \ifcase #1 %
    No citations.
  \or
    (page #2).%
  \else
    (pages #2).%
  \fi%
}
\input scrload

\theoremstyle{definition} 

\newtheorem{thm}{Theorem}[section]      
\newtheorem{definition}[thm]{Definition}
\newtheorem{prop}[thm]{Proposition}
\newtheorem{lemma}[thm]{Lemma}
\newtheorem{corollary}[thm]{Corollary}

\newtheorem{remark}[thm]{Remark}
\newtheorem{example}[thm]{Example}
\newtheorem{conj}{Conjecture}

\newcommand{\ini}{\mathrm{in}}
\newcommand{\vol}{\mathrm{Volume}}
\newcommand{\dirr}{P(\ZZ^2)}
\newcommand{\dirrk}{P(\ZZ^k)}
\newcommand{\dird}{P(\D)}
\newcommand{\Inf}{\mathfrak{Infl}}

\newcommand{\I}{\mathfrak{I}}

\newcommand{\D}{\Delta}
\newcommand{\ZZ}{{\mathbb Z}}
\newcommand{\NN}{{\mathbb N}}

\newcommand{\KK}{{\mathbb K}}
\newcommand{\CC}{{\mathbb C}}
\newcommand{\RR}{{\mathbb R}}

\newcommand{\FF}{{\mathbb F}}
\newcommand{\TT}{{\mathbb T}}
\newcommand{\A}{{\mathcal A}}

\newcommand{\val}{\mathrm{val}}
\newcommand{\Val}{\mathrm{Val}}
\newcommand{\Trop}{\mathrm{Trop}}
\newcommand{\area}{\mathrm{area}}
\newcommand{\mult}{\mathrm{mult}}

\newcommand{\Star}{\mathrm{Star}}

\begin{document}

\title{Tropical approach to Nagata's conjecture in positive characteristic}

\date{\today}

\author[N. Kalinin]{Nikita Kalinin}

\address{Departamento de Matem\'aticas, CINVESTAV
Apartado Postal: 14-740, C.P. 07000,Ciudad de M\'exico, Mexico} 


\email{nikaanspb\{dog\}gmail.com}

\begin{abstract}

Suppose that there exists a hypersurface with the Newton polytope $\Delta$,
which passes through a given
set of subvarieties. Using tropical geometry, we associate a subset of $\Delta$ to each of these subvarieties. We prove that a weighted sum of the volumes of these subsets
estimates the volume of $\Delta$ from below.

As a particular application of our method we consider a planar algebraic curve $C$ which passes through generic points $p_1,\dots,p_n$ with
prescribed multiplicities $m_1,\dots,m_n$. Suppose that the minimal lattice width  $\omega(\Delta)$  of the Newton polygon $\Delta$ of the curve $C$ is at least $\max(m_i)$. Using tropical floor diagrams (a certain degeneration of $p_1,\dots, p_n$ on a horizontal line)  we prove that $$\mathrm{area}(\Delta)\geq \frac{1}{2}\sum_{i=1}^n m_i^2-S,\ \  \text{where } S=\frac{1}{2}\max \left(\sum_{i=1}^n s_i^2 \Big| s_i\leq m_i, \sum_{i=1}^n s_i\leq \omega(\Delta)\right).$$

In the case $m_1=m_2=\ldots =m\leq \omega(\Delta)$ this estimate becomes
$\mathrm{area}(\Delta)\geq \frac{1}{2}(n-\frac{\omega(\Delta)}{m})m^2$. That rewrites as $d\geq (\sqrt{n}-\frac{1}{2}-\frac{1}{2\sqrt n})m$ for the curves of degree $d$.

We consider an arbitrary toric surface (i.e. arbitrary $\Delta$) and our ground field is an infinite field of any characteristic, or a finite field large enough. The latter constraint arises because it is not {\it \`a priori} clear what is {\it a collection of generic points} in the case of a small finite field. We construct such collections for fields big enough, and that may be also interesting for the coding theory. 

\end{abstract}

\keywords {Nagata's conjecture, {\it m}-fold point, floor diagrams, tropical geometry}

\maketitle


\section{Main Theorem and a discussion around Nagata's conjecture}

It is easy to find a polynomial in one variable with prescribed
values at given points. Then, it is not difficult to find a polynomial in
many variables with prescribed values at given points, or to find a polynomial in one
variable with prescribed higher derivatives at given points. Each of
the conditions appearing above imposes one linear constraint on the
polynomial's coefficients. When we have only linear constraints, it is natural to ask whether they are mutually independent. In the above problems it is indeed the case, but in general it is not always true and can be a source of major difficulties.  

Consider the following general question: given  natural numbers  $m_1,m_2,\dots,m_n$ and a set of varieties
$X_1,X_2,\dots,X_n\subset \FF^k$ (where $\FF$ is a field of any characteristic), we are wondering
if there exists a hypersurface $Y\subset \FF^k$ (with a given Newton
polytope $\D$) which passes through each of the $X_i$ with the multiplicity
$m_i$ respectively. This procedure (defining a variety by incidence and tangence relations) helps in constructing concrete examples and counter-examples in the realm of singular varieties. 

This paper promotes the tropical point of view on the above problem. We define
{\it the subsets} $\Inf(X_i)$ of $\Delta$, ``influenced'' by each of the $X_i$. These subsets
can overlap, but no more than $k$ at once (Corollary~\ref{cor_volume}). We mainly concentrate on the case $k=\dim Y+1 = 2$, i.e. $Y$ is a planar algebraic curve and each of the $X_i$ is a point. 

\subsection{Main Theorem}
A vector $(u_1,u_2)\in \ZZ^2$ is {\it primitive} if $\mathrm{gcd}(u_1,u_2)=1$. In particular, a primitive vector $(u_1,u_2)$ cannot be $(0,0)$. Denote by $P(\ZZ^2)$ the set of primitive vectors in $\ZZ^2$.
The lattice width $\omega_u(\D)$ of a polygon $\Delta\subset \ZZ^2$ in a direction $u=(u_1,u_2)\in P(\ZZ^2)$ is $\max\limits_{x,y\in \D}(u_1,u_2)\cdot (x-y)$. 

\begin{definition}
 The minimal lattice width $\omega(\D)$ of a polygon $\Delta\subset \ZZ^2$ is $\min\limits_{u\in P(\ZZ^2)}\omega_u(\D).$
\end{definition}

The following theorem is an application of general methods developed in this article, bred with results of \cite{kalinin}.

\begin{thm}
\label{main_theorem} Let $\FF$ be an infinite field or a field big enough (for details see Lemma~\ref{lemma_detrop}). If $\omega(\Delta)\geq \max(m_i)$ and for
each set of points $p_1,p_2,\dots,p_n\in (\FF^*)^2$ there exists an algebraic  curve $C\subset (\FF^*)^2$ with the Newton polygon $\D$, passing through $p_1,p_2,\dots,p_n$ with multiplicities
$m_1,m_2,\dots,m_n$ correspondingly, then 
\begin{equation}
\label{eq_main}
\mathrm{area}(\Delta)\geq \left(\frac{1}{2}\sum_{i=1}^n m_i^2 \right)- S\big( m_1,m_2,\dots,m_n,\omega(\D)\big).
\end{equation} 
\end{thm}
The correction term $S$ in Eq.\eqref{eq_main} is given by the following formula.
\begin{definition}
For a set $m_1,m_2,\dots,m_n\in\ZZ_{>0}$ we define 
\begin{equation}
\label{eq_correction}
S(m_1,\dots,m_n,M)=\frac{1}{2}\max \left(\sum_{i=1}^n s_i^2\right)
\end{equation} where we maximize by all sets of numbers $\{s_i\}_{i=1}^n$ such that  $0 \leq s_i\leq m_i, \sum_{i=1}^n s_i\leq M$.
\end{definition}
\begin{example}
\label{ex_simplest}
If $m_1=m_2=\dots=m_n=m, M=[\sqrt n]m$, then $S(m_1,\dots,m_n,M)=\frac{1}{2}[\sqrt n]m^2$.
\end{example}

Indeed, if for $a,b\geq 0$ the sum $a+b$ is fixed, then $a^2+b^2$ is bigger when $a$ and $b$ are maximally far from each other. Using this example we obtain the following corollary of Theorem~\ref{main_theorem}.

\begin{corollary}
Under the hypothesis of Theorem~\ref{main_theorem}, if $m_1=m_2=\dots=m_n=m, \omega(\D)=[\sqrt n]m$, then $\mathrm{area}(\Delta)\geq \frac{1}{2}\left(n-[\sqrt n]\right)m^2$.
\end{corollary}

\begin{corollary} Under the hypothesis of Theorem~\ref{main_theorem},  suppose that $$m_1=m_2=\dots=m_n=m\leq \omega(\D).$$ Then, the following inequality takes place: $\area(\D)\geq \frac{1}{2}(n-\frac{\omega(\D)}{m})m^2$. 
\end{corollary}
\begin{proof}
We are seeking for the minimum of $nm^2-2S(m,\dots,m,\omega(\Delta))=\sum_{i=1}^n (m^2-s_i^2)$ under conditions
$\sum s_i\leq \omega(\D),s_i\leq m$. Choose $k,k'\in\ZZ$ such that $\omega(\D)=mk+k', 0\leq k'<m$. We see, using the argument in Example~\ref{ex_simplest}, that the minimum is attained when 
$$s_j=m, \text{if }1\leq j\leq k, \text{ and } 0\leq s_{k+1}=k'<m, \text{ and } s_{j}=0 \text { for } j>k+1.$$

Hence $\sum_{i=1}^n(m^2-s_i^2)\geq nm^2-km^2-{k'}^2$. Therefore, $$\area(\D)\geq\frac{1}{2}\big(nm^2-km^2-{k'}^2\big)\geq \frac{1}{2}\left(n-\frac{\omega(\D)}{m}\right)m^2,$$
because $k+\frac{k'^2}{m^2}\leq k+ \frac{k'}{m}<k+1$.
The equality in the right hand side takes place if $k'=0$.
\end{proof}

\begin{corollary}
If $\Delta$ is the triangle $\mathrm{ConvHull}\big((0,0),(d,0),(0,d)\big)$, then the above corollary gives $d\geq (\sqrt{n}-\frac{1}{2}-\frac{1}{2\sqrt n})m$.
\end{corollary}
\begin{proof}Indeed, $\area(\Delta)=\frac{d^2}{2}$. So, we have $d^2\geq (n-d/m)m^2$. If $d\geq m\sqrt n$, then we are done. Suppose that $d< m\sqrt n$, then $$d^2\geq (n-d/m)m^2\geq (n-\sqrt n)m^2\geq \left(\sqrt n- \frac{1}{2} -\frac{1}{2\sqrt n}\right)^2m^2,$$
because $(\sqrt n- \frac{1}{2} -\frac{1}{2\sqrt n})^2=n-\sqrt n - (1-1/4-1/{4n} -1/{2\sqrt n})\leq n-\sqrt n$.
\end{proof}

\subsection{The idea of proof}
Let $\KK$ be the field of rational functions over $\FF$.
Since each element of $\KK$, except zero, can be written as $$\left\{\sum\limits_{m}^{\infty} c_kt^k|m\in\ZZ,c_k\in\FF\right\},$$  we can define a {\it valuation map} $\val: \KK \to \TT$ by the rule $\val\left(\sum\limits_{m}^{\infty} c_kt^k\right):=-\min\{k,c_k\ne 0\}$ and 
$\val(0):=-\infty$. In other words, $\val(f)$ is minus the order of vanishing of $f$ at zero. There are different extensions of $\KK$, algebraically closed, with surjective map $\val$; one can take a field of power series whose elements converge near zero, etc; see \cite{puiseux, puiseaux2}.

We prove that Theorem~\ref{main_theorem} holds over
this valuation field $\KK$. We use the nature of a
singular point's influence on the Newton polygon of a curve (this ``influence'' means that a part  $\Inf(p_i)$ of $\Delta$ corresponds to each point $p_i$, see Section~\ref{sec_influence}) and
 tropical floor diagrams \cite{floor2,floor1}. Tropical floor diagrams illustrate the process of a degeneration of the points $p_1,\dots,p_n$ on a horizontal line, in a sense it is a tropical version of the Horace method \cite{evain}. While degenerating
$p_1,p_2,\dots, p_n$ onto a line, we see the following behavior (Figure \ref{fig:main}) of
the points on the tropical picture. Each point of multiplicity $m_i$ splits into two parts $m_i=s_i+r_i$, such that $\sum_{i=1}^n s_i\leq \omega(\D)$. Furthermore, we choose a part of $\Inf(p_i)\subset\Delta$ for each $i=1,\dots,n$; these parts do not intersect, and the area of such a part for a point $p_i$ is at least $\frac{1}{2}(m_i^2-s_i^2)$. 

Then we prove {\it Detropicalization lemma}. It says that if Theorem~\ref{main_theorem} holds over $\KK$ and does not hold over $\FF$, then there exists a non-zero polynomial of bounded degree which has all points of $\FF$ as its roots. This implies that there exists a constant $N\in\NN$ such that if the cardinality of $\FF$ is at least $N$ (which is always the case if $\FF$ is infinite), then Theorem~\ref{main_theorem} holds for $\FF$. This is a natural restriction, because in small fields we
cannot find a sufficiently generic collection of points. The constant $N$, then, depends on $\max(m_i),\D$, and $\mathrm{char}(\FF)$. This reasoning could be of a particular interest to coding theory, see Section \ref{codetheory}.

\subsection{Nagata's conjecture}
 
Let us fix a field $\FF$. For a point $p = (x_1,y_1)\in \FF^2$ we
denote by $I_p$ the ideal of the point $p$, namely $I_p=\langle x-x_1,y-y_1\rangle$.

\begin{definition}
Consider an algebraic curve $C$ given by an equation $F(x,y)=0, F\in \FF[x,y]$. We say
that $p$ is {\it of multiplicity at least $m$} for $C$ (and write $\mu_p(C)\geq m$), if $F\in
(I_p)^m$.
\end{definition}

In the most non-degenerate case,  ``$p$ is a point of multiplicity $m$
on $C$'' means that there are $m$ branches
of $C$ passing through $p$. For fields of characteristic zero, two following conditions are equivalent: 1) $F\in (I_p)^m$, and 2) all the partial derivatives of $F$ up to order $m-1$ vanish at $p$.

\begin{example}
\label{example_sing}
Consider a planar algebraic curve $C$ of degree $d$ given by an
equation $F(x,y)=0$, where  $$F(x,y)=\sum\limits_{i,j\geq 0, i+j\leq
  d}a_{ij}x^iy^j.$$ The point $p=(0,0)$ is of multiplicity $m$ for $C$ if
and only if for all $i,j\geq 0$ with $i+j<m$ we have $a_{ij}=0$.  As a consequence,
for each point $p\in \FF^2$ the condition $\mu_p(C)\geq m$ can be
rewritten as a system of $\frac{m(m+1)}{2}$ linear equations in the
 coefficients $\{a_{ij}\}$ of $F$.
\end{example}

Let $p_1,\dots, p_n$ be a collection of $n>9$ points in $\FF^2$ and  $m_1,\dots, m_n\in \NN$. We are looking for the
minimal degree $d_{min}$ of an algebraic curve passing through $p_1,\dots, p_n$ with
multiplicities at least $m_1,\dots, m_n$ respectively.

One can naively calculate the expected dimension
$\mathfrak{edim}(d,m_1,\dots,m_n)$ of the space $\mathfrak{S}$ of the
curves of degree $d$ satisfying the hypothesis above. Indeed, each singular point imposes $\frac{m(m+1)}{2}$ constraints on the coefficients of
the curve equation, therefore, $$\mathfrak{edim}(d,m_1,\dots,m_n) = \max\left(-1, \frac{d(d+3)}{2}-\sum\limits_{i=1}^{n} \frac{m_i(m_i+1)}{2}\right).$$

The actual dimension of $\frak{S}$  is always at least the
expected one, because all the constraints are linear. However, sometimes even for a generic choice of the set of points $p_1,p_2,\dots,p_n$ the
actual dimension is strictly greater than the expected one. 

As a reasonable estimate for $d_{min}$,  Nagata's conjecture claims:

\begin{conj} If points $p_1,\dots,p_n\in \mathbb P^2, n>9$ are chosen generically and $d\leq\frac{1}{\sqrt{n}}\sum\limits_{i=1}^{n}
  m_i$, then $\dim\frak{S}=-1$. In other words, $d_{min}>\frac{1}{\sqrt{n}}\sum\limits_{i=1}^{n}
  m_i$.
\end{conj}

\begin{example}
\label{ex1}
Let us consider two points $p_1,p_2$. The minimal degree of a curve
passing through $p_1,p_2$ with multiplicities $m_1,m_2$ is
$m_1$, if $m_1\geq m_2$:  it is the line passing through $p_1$ and $p_2$ taken
with multiplicity $m_1$. So the inequality $d_{min}\geq
\frac{m_1+m_2}{\sqrt{2}}$ in Nagata's conjecture is not satisfied
as long as $m_2>m_1(\sqrt{2}-1)$. For five generic points $p_i$ with multiplicities $m_i,i=1,\dots,5$, one can draw a conic through $p_1,p_2,p_3,p_4,p_5$, and take it with multiplicity $\max (m_i)$. In the case all $m_i=m$, this example also violates the inequality in Nagata's conjecture. A classification for the case of less than ten points can be found in (\cite{MR2098342}, Example~2.4, \cite{ciliberto}, Proposition~5.8, Theorem~5.9).
\end{example}

The case $n=l^2$ had been proven by Nagata himself \cite{nagata}. These days, even the case $n=10$
and $m_1=m_2=\dots=m_{10}=m$ is under
exhaustive study (\cite{ten}), but has not yet been
proven. The similar questions in higher dimensions are widely open
(cf. \cite{bocci2004special},\cite{dumnicki2014linear}). The pictures appeared in our approach are somewhat similar to those in \cite{MR3003317}, though the relation is not direct.

Historically Nagata's conjecture  appeared
as a tool (with $n=16$) to disprove Hilbert 14th problem. Related to Nagata's conjecture is the
Segre-Harbourne-Hirschowitz's conjecture which basically says that if the
expected dimension $\mathfrak{edim}$ of $\frak{S}$ is not equal to the actual one,
then the linear system $\frak{S}$ contains a rational curve in its
base locus. The reader is kindly referred to look into surveys
\cite{ciliberto, survey1,harbourne,miranda} for an
introduction to Nagata's conjecture and related topics.

In view of Theorem \ref{main_theorem} the following three results should
be mentioned:

{\bf Theorem} (\cite{xu}, Xu). {\it If $C$ is a reduced and irreducible curve of degree $d$
passing through generically chosen
points $p_1,p_2,\dots,p_n\in \CC P^2$ with multiplicities
$m_1,m_2,\dots,m_n$ respectively, then the estimate $d^2\geq \sum_{i=1}^nm_i^2 -
\min (m_i)$ holds.}

Xu's theorem can be {\it verbatim} extended to curves with arbitrary Newton polygons, and to reducible and non-reduced curves. But unlike Xu's theorem, which requires characteristic zero, in Theorem~\ref{main_theorem} we consider curves defined over fields of any characteristic. 

{\bf Theorem} (\cite{alexander}, Alexander, Hirschowitz). {\it The dimension of
the space of degree $d>2$ hypersurfaces in $\CC P^k (k\geq 3)$, passing through
generic points $p_1,p_2,\dots,p_n$ with multiplicities
$m_1=\dots=m_n=2$, is the expected one except for the cases
$(k,d,n)=(2,4,5),(3,4,9),(4,4,14),(4,3,7)$.}

Using the classification of tropical singular surfaces in \cite{markwig2}, we give a sketch of a proof that the volume $V$ of the Newton polytope of a surface in $\CC P^3$ with $n$ two-fold points in {\bf general} position satisfies $n\leq 2V$. Using the above theorem we can obtain a better estimate, see Remark~\ref{rem_surface} for details.

{\bf Theorem} (\cite{alexander2}, Alexander, Hirschowitz). {\it For each field $\FF$,
  the dimension of degree $d$ hypersurfaces in $\FF P^k$ passing through
generic points $p_1,p_2,\dots,p_n$ with multiplicities
$m_1,m_2,\dots,m_n$ is the expected one if $d\gg\max m_i$}.

We expect that our approach can be extended to the cases $k\geq 3$ and $m_i>2$. Such an extension will lead to an explicit degree estimate in these cases.

Research is supported by the grant 168647 (PostDoc.Mobility) of the Swiss National Science Foundation. I would like to thank the anonymous reviewers for numerous suggestions which improved this paper a lot.

\section{Preliminaries in tropical geometry}
\label{sec_general}

In this section we recall some definitions and set up the notation. We
  refer the reader to \cite{BIMS},\cite{maclagan2015introduction} for a general introduction to
  tropical geometry. 


Let $\TT$ denote $\RR\cup \{-\infty\}$, and $\KK$ be a field with a
valuation map $\val:\KK\to\TT$. We use the convention $\val(a+b)\leq
\max(\val(a),\val(b)), \val(0)=-\infty$. Usually $\TT$ is called the {\it tropical
  semi-ring}.  

Consider a hypersurface $Y\subset \KK^k$. Let $Y$ be given by an equation
$$F(x_1,x_2,\dots,x_k)=0,$$ $$F=\sum_{I\in\A}c_Ix^I,
 I=(i_1,i_2,\dots,i_k), c_I\ne 0.$$ In such case $\D=\mathrm{ConvexHull}(\A)$ is called
{\it the Newton polytope} of $Y$.

The Newton polytope of $F$ is provided with a subdivision deined by the coefficients of $F$.
Namely, consider the extended Newton polytope of $Y$,
$$\widetilde\D=\mathrm{ConvexHull}\{(I,x)\in \ZZ^k\times \TT|I\in \A, x\leq \val(c_I))\}.$$
The projection of the faces of the extended Newton polytope $\widetilde\D$ onto 
the Newton polytope $\D$ defines a subdivision of $\D$.

We give a definition of the tropicalization of $Y$, based on its
equation $F(x)=\sum_{I\in\A}c_Ix^I$. For a {\it weight $\omega=(w_1,w_2,\dots,w_k)\in \TT^k$} we consider
the {\it weight function} $$\omega (cx_1^{i_1}x_2^{i_2}\dots x_k^{i_k})
:= \val(c)+i_1w_1+i_2w_2+\dots+i_kw_k.$$ Then we define
initial part $\ini_\omega(F)$ of $F$ as the 
$\omega$-maximal part of $F$. Namely, we find $W$, the maximal weight (with respect to $\omega$) among monomials of $F$. Then, by definition, $\ini_\omega(F)$ is the sum of all monomials of $F$ with weight $W$ with respect to $\omega$. Finally, we define
$\Trop(Y)\subset \TT^k$ to be the set of all weights $\omega$ such that $\ini_\omega(F)$
is not a monomial. This is same as define $\Trop(Y)$ as the corner locus of the function $$\Trop(F)(\omega) = \max_{I\in\A}(\val(c_I+ I\cdot\omega)).$$


The set $\Trop(Y)$ is a polyhedral complex. Indeed, for each subset $S$ of the set of monomials of $F$ we can consider the set $S^*$ of $\omega\in\TT^k$ such that the set of monomials in $\ini_\omega(F)$ is $S$. One can prove that each of sets $S^*$ is convex. Therefore this defines a polyhedral subdivision of $\TT^k$, whose cells are parametrized by subsets $S$ of the set of monomials of $F$. Adopting this point of view we see that $\Trop(Y)$ is the union of cells in the above subdivision, which correspond to $S$ with $|S|\geq 2$.

This ({\it Gr\"obner}) subdivision of $\TT^k$ is related with the subdivision (described earlier) of the Newton polytope $\Delta$ in the following way. A
point $I\in\D$ is a vertex of the subdivision of $\Delta$ if there exists such a
weight $\omega_0\in\TT^k$ that $\ini_{\omega_0}(F) = c_Ix^I$. So we say that the cell $\{\omega'\in\TT^k | \ini_{\omega_0}(F)=\ini_{\omega'}(F)=c_Ix^I\}$ is dual to the vertex $I$ in the subdivision of $\D$.

An interval
$I_1I_2$ between two vertices $I_1,I_2\in\D$ is an edge of the subdivision of $\Delta$ if there exists a
weight $\omega_0$ such that $\ini_{\omega_0}(F) = \sum_{I\in J}
c_Ix^I$ where the convex hull of $J$ is the interval $I_1I_2$. Again, we say that the cell $\{\omega'\in\TT^k| \ini_{\omega'}(F)=\ini_{\omega_0}(F)\}$ is dual to the interval $IJ$ of the subdivision of $\Delta$. We further define this duality between the cells of dimension $i$ in the subdivision of $\Delta$ and cells of dimension $k-i$ in the subdivision of $\TT^k$ similarly.

We can summarize the above arguments as follows.
\begin{remark}\label{rem_duality}
Each cell of the subdivision of $\D$ is of the type 
$$\Delta_\omega = \mathrm{ConvexHull(support(}\ini_\omega(F)))$$ for some
$\omega\in \TT^k$. We recall that the support of a polynomial is the set of multipowers of its monomials. For example, $\mathrm{support}(x^2+t^{-2}xy+3y^3)$ is the set $\{(2,0),(1,1),(0,3)\}$.
\end{remark}

\begin{remark}
If $Y\subset \KK^k$ is a hypersurface, then $\Trop (Y)\subset \TT^k$ is a polyhedral complex of
codimension one. For each cell
$\Delta_\omega\subset\D$ of the subdivision of the Newton polytope $\Delta$ of $Y$ we define $d(\Delta_\omega) =\{\omega'\in \TT^k|
\Delta_\omega = \Delta_{\omega'}\}$. 

This map $d$ provides the following correspondence: the vertices of the subdivision of $\D$ correspond to the
connected components of the complement of $\Trop(Y)$ in $\TT^k$, the edges of the
subdivision correspond to the faces of $\Trop(Y)$ of the maximal
dimension, the two-cells of the subdivision correspond to the faces of
codimension one in $\Trop(Y)$, etc. A particular example of such a duality for $k=2$ is presented in Figure~\ref{governorship}. Since $d$ is a bijection between cells, abusing notation we also write $\Delta_\omega =d(\{\omega'\in \TT^k|
\Delta_\omega = \Delta_{\omega'}\})$.
\end{remark}

\begin{definition}
If $X\subset \KK^n$ is a variety of higher codimension, we define its tropicalization $\Trop(X)$ as follows. Let $I$ be the ideal of $X$. Let $\ini_\omega(I)$ be the ideal generated by the elements $\ini_\omega(f),f\in I$. Then, by definition, $\omega\in\Trop(X)$ if and only if $\ini_\omega(I)$ is monomial free.
\end{definition}
 A proof that $\Trop(X)$ is a polyhedral complex repeats the above arguments for the case of hypersurface, see \cite{maclagan2015introduction} for details. 

\section {An estimate of a singular point's influence on the Newton
  polygon of a curve}
\label{sec_influence} 
In this section we cover facts from \cite{kalinin} that we need. Let $C$ be a curve over a valuation field $\KK$ with the Newton polygon $\D$ such that $\omega(\D)\geq m$. Let $Q\in\Trop(C)$.

\begin{figure}[htb]
\begin{center}
\begin{subfigure}[b]{0.3\textwidth}
\begin{tikzpicture}[scale=0.3]
\draw [very thin, gray] (0,-1) grid (15,11);
\draw[very thick] (1,10)--(10,10)--(12,9)--(14,6)--(14,0)--(2,1)--cycle;
\draw[very thick]
(2,6)--(3,6)--(4,5)--(6,8)--(8,9)--(10,8)--(12,7)--(13,6)--(12,6)--(10,5)--(8,3)--(6,4)--(4,3)--(3,5)--(2,6);

\draw[very thick] (3,6)--(3,5);
\draw[very thick] (4,5)--(4,3);
\draw[very thick] (6,8)--(6,4);
\draw[very thick] (8,9)--(8,3);
\draw[very thick] (10,8)--(10,5);
\draw[very thick] (12,7)--(12,6);
\end{tikzpicture}
\end{subfigure}
\quad\quad\quad
\begin{subfigure}[b]{0.3\textwidth}
\begin{tikzpicture}[scale=0.3]
\draw [very thin, gray] (0,0) grid (16,12);
\draw[very thick] (2,0)--(13,0)--(16,2)--(16,9)--(8,12)--(1,8)--cycle;
\draw[very thick]
(3,7)--(4,7)--(8,11)--(8,12)--(10,11)--(11,10)--(12,8)--(13,8)--(16,5)--(13,5)--(12,4)--(12,2)--(10,2)--(10,0)--(9,0)--(7,2)--(5,5)--cycle;

\draw[very thick] (4,7)--(4,6);
\draw[very thick] (5,8)--(5,5);
\draw[very thick] (8,11)--(10,11);
\draw[very thick] (7,10)--(11,10);
\draw[very thick] (12,8)--(12,4);
\draw[very thick] (13,8)--(13,5);
\draw[very thick] (14,7)--(14,5);
\draw[very thick] (10,2)--(12,4);
\draw[very thick] (11,2)--(12,3);
\draw[very thick] (7,2)--(10,2);
\draw[very thick] (8,1)--(10,1);

\draw[very thick] (5,8)--(6,10)--(7,10);
\draw[very thick] (11,10)--(12,10)--(12,8);
\draw[very thick] (5,5)--(4,3)--(7,2);
\end{tikzpicture}
\end{subfigure}
\end{center}
\begin{center}
\begin{subfigure}[h]{0.3\textwidth}
\begin{tikzpicture}[scale=0.3]
\path (0,0);
\begin{scope}[shift={(0,-5.5)}]
\draw (0,0) node {$\bullet$};
\draw (0,0) node[above] {$Q$};

\draw[very thick] (-7,0)--(5,0);
\draw[very thick] (2,2)--(1,0)--(2,-1);
\draw[very thick] (4,2)--(3,0)--(4,-2);
\draw[very thick] (6,1)--(5,0)--(5,-1);
\draw[very thick] (-3,2)--(-2,0)--(-3,-2);
\draw[very thick]  (-4.75,0.5)--(-4,0)--(-3.5,-2);
\draw[very thick]  (-5.5,0.5)--(-6,0)--(-6,-0.5);
\draw[very thick] (-7,1)--(-7,0)--(-7.5,-0.5);
\end{scope}
\end{tikzpicture}
\vspace{23pt}
\caption{if $Q$ is not a vertex}
\label{governorship_1}
\end{subfigure}
\quad\quad\quad
\begin{subfigure}[ht]{0.3\textwidth}
\begin{tikzpicture}[scale=0.3]
\path (-5,0);
\begin{scope}[shift={(26mm,0)}]
\draw (0,0) node {$\bullet$};
\draw (0,0) node[above right] {$Q$};
\draw[very thick] (0,0)--(0,2)--(1,3);
\draw[very thick] (-1,3)--(0,2)--(0,4)--(-1,4);
\draw[very thick] (0,4)--(0.5,5);
\draw[very thick] (0,0)--(2,1)--(3,1);
\draw[very thick]  (2,1)--(2,2);
\draw[very thick]  (0,0)--(-2,2)--(-2,3);
\draw[very thick]  (-2,2)--(-4,3);
\draw[very thick] (0,0)--(3.5,0)--(3.5,0.5);
\draw[very thick] (4,-0.5)--(3.5,0)--(5,0)--(5,-1);
\draw[very thick]  (6,1)--(5,0)--(6,0)--(7,1);
\draw[very thick] (6,0)--(6,-1);
\draw[very thick]  (0,0)--(-3,0)--(-4,1);
\draw[very thick] (-4,-1)--(-3,0)--(-5,0)--(-5,1);
\draw[very thick] (-5,0)--(-6,-1);
\draw[very thick] (0,0)--(-3,-2)--(-4,-1.5);
\draw[very thick] (-3,-2)--(-4,-5);
\draw[very thick]  (0,0)--(0,-2.5)--(1,-2.5);
\draw[very thick] (-1,-3.5)--(0,-2.5)--(0,-4)--(1,-4);
\draw[very thick]  (-1,-5)--(0,-4)--(0,-5);
\draw[very thick]  (0,0)--(1.5,-1.5)--(2.5,-1.5);
\draw[very thick] (1.5,-2.5)--(1.5,-1.5)--(2.5,-2.5)--(3.5,-2.5);
\draw[very thick] (2.5,-3.5)--(2.5,-2.5);
\end{scope}
\end{tikzpicture}
\caption{if $Q$ is a vertex}
\label{governorship_2}
\end{subfigure}
\end{center}

\caption{If $Q$ is not a vertex of $\Trop(C)$ (left column), then the
  collection $\I(Q)$ of vertices consists of all the vertices of
  $\Trop(C)$ lying on the extension of the edge through $Q$.  If $Q$ is a vertex of $\Trop(C)$ (right column), then $\I(Q)$ is the set of the vertices on the extensions of all the edges through $Q$. In each case the
  corresponding set $\Inf(Q)$ of faces of the subdivision of
  $\Delta$, the ``region of influence'' of $Q$, is drawn at the top.}
\label{governorship}

\end{figure}

\begin{definition}\label{defi_lq}Let $l_Q(u)$ be the line through $Q$ in the direction $u\in\dirr$. 
Take the connected component, containing $Q$, of the intersection
  $\Trop(C)\cap l_Q(u)$. We call this component {\it the long edge
    through $Q$ in the direction $u$} and denote it by $E_Q(u)$. 
\end{definition}

\begin{definition}
\label{def_I}
For each $u\in\dirr$ we denote by $\I_Q(u)$ the set of vertices of $\Trop(C)$ which belong
to the long edge $E_Q(u)$. Define $\I(Q)=\bigcup_{u\in\dirr}\I_Q(u)$.
\end{definition}

Note that $\I(Q)$ is not a multiset; it contains only one copy of
$Q$.  Examples of $\I(Q)$ are presented in Figure \ref{governorship}. 
On the left we see one long edge $E_Q((1,0))$ and
$\I(Q)$ consists of 7 vertices, and above we see 7 corresponding
faces in the subdivision of $\Delta$. On the right,  we see long edges
$E_Q((1,0)),E_Q((0,1)),E_Q((-1,1))$. Each of the long edges 
$E_Q((1,2))$ and $E_Q((-3,-2))$ consists of only one edge. 

\begin{definition}[\cite{kalinin}]
For a point $Q\in \Trop(C)$ we define the {\it region of influence of $Q$} $$\Inf(Q)=\bigcup\limits_{V\in \I(Q)}d(V),$$ the union of the faces of the Newton polygon of $\Trop(C)$, dual to the vertices in $\I(Q)$.
\end{definition}

\begin{definition}
\label{def_influence_edge}
For a point $Q
\in \Trop(C)$ which is not a vertex of $\Trop(C)$,
we define $$\area(\Inf(Q))=\sum\limits_{F\in\Inf(Q)}\area(F),$$
the sum of the areas of faces $F$ in the region of influence of $Q$, see Figure~\ref{governorship}.
\end{definition}

Note that $\area(\Inf(Q))$ depends only on $\Trop(C)$ and does not depend on a
particular choice of an equation defining $C$. Also, if $Q$ belongs to an edge $E$ of $\Trop(C)$ and $Q$ is not a vertex of $\Trop(C)$,
then $\I_Q(u)=\I(Q)$ where $u$ is the direction of $E$. Indeed, for any other direction $v$ not collinear to $u$, 
the connected component of $Q$ in the intersection $\Trop(C)\cap
l_Q(v)$ is just $Q$.

Recall that if $Q$ is a vertex of  $\Trop(C)$, then $d(Q)$ is a face dual to $Q$ in the
subdivision of $\Delta$.

\begin{definition}
\label{def_influence_vertex}
If $Q$ is a vertex of $\Trop(C)$, we define \begin{align*}\area(\Inf(Q))&=\sum\limits_{F\in \Inf(Q)}\area(F)+\area(d(Q)),\\
\area^{*}(\Inf(Q))&=\sum\limits_{F\in \Inf(Q)}\area(F).\end{align*}  
\end{definition}

\begin{thm} [\cite{kalinin}, Lemma 2.8, Theorems 1,2] 
\label{infl_theorem} Suppose that a point $p=(p_1,p_2)\in(\KK^*)^2$ is of multiplicity $m$ for
this curve $C$, $P=\Val(p)=(\val(p_1),\val(p_2))$.  Suppose also that the Newton polygon $\D$ of $C$ satisfies $\omega(\D)\geq m$. Then, 
\begin{equation}
\area(\Inf(P))\geq \frac{m^2}{2}.
\end{equation}
\end{thm}

If the point $Q$ in Figure~\ref{governorship} is as in this theorem, then the sum of the
  areas of the faces is at least $\frac{1}{2}m^2$
  in (A) and at least $\frac{3}{8}m^2 (= \area^{*}(\Inf(Q)))$ in (B) (but we will not consider $\area^*$ in this article).

\begin{example}
Consider a curve $C$ given by the equation $(x-1)^k(y-1)^{m-k}=0$,
take $p=(1,1)$. Clearly, $\mu_{p}(C)=m$, but the Newton polygon $\D$ of  $C$ violates the condition $\omega(\D)\geq m$, and the inequality $\area(\Inf(\Val(p)))=2k(m-k)\geq \frac{m^2}{2}$ does not hold except for the case $k= m/2$. 
\end{example}

\begin{lemma}[\cite{kalinin}, Lemma 1.25]
\label{lemma_width}
If $\mu_{(1,1)}(C)=m$ and for the Newton polygon $\D$ of $C$ we have $\omega_u(\D)=m-a$ for some $a>0, u=(u_1,u_2)\in P(\ZZ^2)$, then $C$
contains a rational component through $(1,1)$ parametrized as $(s^{u_1},s^{u_2})$. 
\end{lemma} 

Note that the tropicalization of such a component is the straight line though $(0,0)$ in the direction $(u_1,u_2)$. We will use this lemma in Corollary~\ref{cor_first}.

\begin{lemma}[\cite{kalinin}, Lemma 2.8,  Lemma 5.20]
\label{lemma_prep}
Let $\mu_p(C)\geq m, p=(x_1,x_2)\in(\KK^*)^2$, denote $P=(\val(x_1),\val(x_2))$. Suppose that $P$ is a vertex of
$\Trop(C)$ and $\omega_{u}(d(P))=a\leq m$ for some direction $u$.
Then both sides of $d(P)$, perpendicular to $u$, have length at least $m-a$, and
\begin{equation}
\sum\limits_{V\in
  \I_P(u),V\ne P}\area(d(V))\geq
\frac{1}{2}(m-a)^2.
\end{equation}
 \end{lemma}

\begin{example}
Suppose that $u=(1,0)$. Note that in this case  $a$ is the length of the projection of $d(P)$ onto the $x$-axis. Recall that $\I_P((1,0))$ is the set of vertices of $\Trop(C)$, lying in the connected component of $P$ in the intersection
of $\Trop(C)$ with the straight horizontal line through $P$. Figure~\ref{est41} illustrates the set of dual faces $d(V)$ to the vertices $V$ in $\I_P((1,0))$. Since the long edge through $P$ is horizontal, all the edges separating  faces in $\Inf(P)$ in Figure~\ref{est41} are vertical.
\end{example}

\begin{figure} [htbp]
\begin{center}
\begin{tikzpicture}[scale=0.5]
\draw (0,4)--(0,9)--(1,10)--(2,10)--(8,7)--(8,2)--(5,1)--(3,1)--cycle;
\draw [fill=gray](0,4)--(-1,4)--(-1,7)--(0,9)--cycle;
\draw [fill=gray](-1,4)--(-2,2)--(-2,3)--(-1,7)--cycle;
\draw [fill=gray](-2,2)--(-3,2)--(-2,3)--cycle;
\draw[<->](0,6)--(8,6);
\draw (4,6.5) node{$a$};
\draw (1.5, 8) node{$\geq m-a$};
\draw (6.5,3) node{$\geq m-a$};
\draw [fill=gray](8,7)--(9,7)--(9,3)--(8,2)--cycle;
\draw [fill=gray](9,7)--(10,8)--(10,6)--(9,3)--cycle;
\draw [fill=gray](10,8)--(11,5)--(10,6)--cycle;
\draw (3,3.5) node{$d(P)$};
\draw(0,4) node[below]{$L$};
\draw(0,9) node[above]{$M$};
\draw(8,7) node[above]{$N$};
\draw(8,2) node[below]{$K$};
\end{tikzpicture}
\qquad
\caption {Dual picture to a singular point $P$ on an edge. Since
  $\omega_{(1,0)}(d(P))=a$, the
  lengths of $LM$ and $NK$ are at least $m-a$ (Lemma~\ref{lemma_prep}).
The set $\bigcup d(V)$ for $V\in \I_{P}((1,0)), V\ne P$ is colored.
The sum of the areas of the colored faces is at least
$\frac{1}{2}(m-a)^2$.}
\label{est41}
\end{center}
\end{figure}
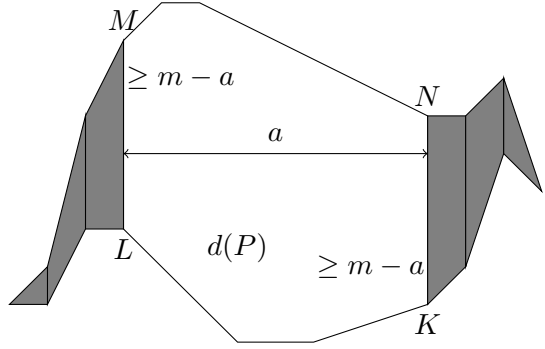

\begin{remark}
\label{rem_length}
Lemma~\ref{lemma_prep} holds in the degenerate case ($a=0$), too: if $P$ belongs to an edge $E$ of $\Trop(C)$, then the dual edge $d(E)$ of the subdivision of the Newton polygon of $C$ has lattice width at least $m$ (Theorem~1 in \cite{kalinin}).
\end{remark}

\begin{remark}
Consider a planar tropical curve $H$. Suppose that this curve is not a usual line. Then the Newton polygon $\Delta$ of $H$ is two-dimensional. The intersection of $H$ with a usual line in direction $u\in P(\ZZ^2)$ is equal to $\omega_u(\Delta)$. Hence, if the intersection of $H$ with each usual line of rational direction is at least $m>0$, then its self-intersection $H\cdot H$ is at least $\frac{3}{4}m^2$ which is the minimal double area of a polytope $\D$ with $\omega(\D)\geq m$, for details and references see \cite{kalinin}. 

The same reasoning works in any dimension. If the intersection of a hypersurface $H\subset \TT P^k$ with any usual line of rational direction is at least $m$, then the minimal lattice width $\omega(\D)$ of the Newton polytope $\D$ of $H$ is at least $m$ and therefore the self-intersection $H^k$ is at least $cm^k$ where $c$ is a constant which depends only on $k$. However, the best value of $c$ in the inequality $\vol(\D)\geq c \omega(\D)^k$ is not known. The following question in codimension two, therefore, is the simplest one possessing no estimate at all (because there is no notion like Newton polytope which keeps track of degree and  self-intersection at the same time).  
\end{remark}

\begin{conj} 
Consider a tropical (i.e. balanced along dimension one faces) two-dimensional fan $L$ in $\RR^4$. Suppose that $L$ is not an affine Euclidean plane. Suppose that the stable tropical intersection of $L$ with each plane of rational slope is at least $m$. Then there exists a constant $c$ such that $L\cdot L\geq cm^2$ in this case, and $c$ does not depend on $m$.
\end{conj}

\section{Influenced subsets in the Newton polytope}
In this section we generalize the definitions of influenced sets in the Newton polygon, given in the article \cite{kalinin} (see Section~\ref{sec_influence} for recap). 
Also we discuss here the notion of a set of points in $\ZZ^k$ in {\it tropical general position with
  respect to a polytope} $\D$.
  
Let $Y$ be a hypersurface in $\KK^k$ with Newton polytope $\D$. In this subsection, for a given subvariety $X\subset Y$, we define the set $\I^\D(\Trop(X))$ of
vertices of $\Trop(Y)$ and the subset $\Inf(\Trop(X))\subset \D$.

\begin{definition}We denote by $\dirrk$ the set of all primitive non-zero vectors in
  $\ZZ^k$.
An affine hyperplane with a normal direction $u\in \dirrk$ is a set $\{x\in \RR^k | u\cdot x=c\}$ with some $c\in\RR$. 
\end{definition}

Let $Q$ be a subset of $\Trop(Y)$.

\begin{definition}
\label{def_star2}
Let $l_{Q}(u)$ be the affine hyperplane in $\RR^k$ with normal direction $u$, containing the set $Q$, if such a hyperplane exists, and $l_{Q}(u)=\varnothing$, otherwise. 
Let $\dird\subset \dirrk$ be the set of the primitive vectors $\{\overline{IJ}| I,J\in\Delta\}$ between the lattice points in $\Delta$. Define the star of $Q$ (with respect to $\Delta$) as $$\Star^\D(Q)=\bigcup_{u\in\dird}l_{Q}(u).$$  
The connected component of $Q$ in the intersection $\Trop(Y)\cap \Star^\D(Q)$ is called {\it the star $\Star^\D_Y(Q)$ of $Q$ in $\Trop(Y)$}. 
\end{definition}

\begin{example} Let $Y\subset (\KK^*)^2$ be an algebraic curve whose Newton polygon is $\Delta$.
\begin{itemize} 
\item If $Q$ is a vertex of $\Trop(Y)$, then $\Star^\D_Y(Q)$ is the connected component of $Q$ in the intersection of $\Trop(Y)$ with the union of the lines spanned by the edges of $\Trop(Y)$ through $Q$, see example in Figure~\ref{governorship}, (B).  
\item If $Q\in\Trop(Y)$ is not a vertex of $\Trop(Y)$, then $\Star^\D_Y(Q)$ is the connected component of $Q$ in the intersection of $\Trop(Y)$ with the line spanned by the unique edge of $\Trop(Y)$ through $Q$, Figure~\ref{governorship}, (A).
\end{itemize}
\end{example}

\begin{definition}
\label{def_verticesstar}
Let $\I^\D(Q)$ be the set of the vertices of $\Trop(Y)$ which belong to $\Star^\D_Y(Q)$.
\end{definition}

We provide each point in $\Star^\D(Q)$ with a multiplicity corresponding to the codimension of its stratum. 
\begin{definition}
\label{def_multpoint}
For a point $V\in \Star^\D(Q)$ the natural number $\mult_Q(V)$ is the dimension of the linear span of the directions $u\in\dird$ such that the affine hyperplane through $V$ with the normal direction $u$ contains $Q$.
\end{definition}

\begin{example}
\label{ex_multi}
If $\D\subset \ZZ^2$ and $Q$ is a point, then $\Star^\D(Q)$ is a union of intervals emanating from $Q$. In this case $\mult_Q(Q)=2$ and $\mult_Q(V)=1$ for $V\in \Star^\D(Q), V\ne Q$.
\end{example}

Each tropical variety $\Trop(X)$ is naturally decomposed into
vertices, edges, faces, etc, because $\Trop(X)$ is a subcomplex of the Gr\"obner complex (see Section~\ref{sec_general}). So, we present $\Trop(X)=\bigcup X^i$ as a union of cells which we denote by $X^i$. Recall that if $X$ is a hypersurface, then each cell of $\Trop(X)$ is an equivalence class of some $\omega\in\Trop(X)$, with the equivalence relation $\omega\sim\omega'$ iff $\D_\omega=\D_{\omega'}$, see Remark~\ref{rem_duality}. Let $X\subset Y$, then $\Trop(X)\subset\Trop(Y)$.
\begin{definition}
\label{def_unionstar}
Define $\I^\D(\Trop(X))=\bigcup\I^\D(X^i)$. We define the {\it star} of the variety $\Trop(X)$ as $$\Star^\D(\Trop(X)) =\bigcup \Star^\D(X^i),\ \ \Star^\D_Y(\Trop(X)) =\bigcup \Star^\D_Y(X^i).$$
So, we take all the cells of $X$, draw the star for each of them, and take the union of these stars. 
\end{definition}

\begin{definition}
\label{def_multall}
For a vertex $V\in \I^\D(\Trop(X))$ we define its multiplicity $\mult_{\Trop(X)}(V)$ as $$\mult_{\Trop(X)}(V)=\max\limits_{X^i}\mult_{X^i}(V),$$ i.e. we take the maximum of the multiplicities of $V$ with respect to the cells in the natural cell decomposition of $\Trop(X)$.
\end{definition}

\begin{definition} Let $X\subset Y\subset \KK^k$ be algebraic varieties, $Y$ be a hypersurface with the Newton polytope $\D$.
The distinguished domain $\Inf(X)$ in $\D$, corresponding to
$X$, is $$\Inf(X)=\bigcup_{V\in\I^\D(\Trop(X))} d(V),$$ where $d(V)$ is the
cell (of the maximal dimension) of $\D$, dual to the vertex $V$ of $\Trop(Y)$. For $\I^\D(\Trop(X))$ see Definitions~\ref{def_verticesstar},~\ref{def_unionstar}.
\end{definition}
Note that $\Inf(X)$ depends only on $\Trop(X)$, so we will write $\Inf(\Trop(X))$.
\begin{definition}
\label{def_influencegeneral}
By $\vol(\Inf(\Trop(X)))$ we denote the sum of volumes (with multiplicities, see Definition~\ref{def_multall}) of the cells in the subdivision of $\Delta$, dual to the vertices in $\I^\D(\Trop(X))$, i.e. $$\vol(\Inf(\Trop(X))) =
\sum\limits_{V\in\I^\D(\Trop(X)) } \mult_{\Trop(X)}(V) \cdot\vol(d(V)).$$
\end{definition}

\begin{example}
\label{ex_multall}
Refer to Example~\ref{ex_multi}. Consider the two-dimensional case, $X=(x_1,x_2)\in(\KK^*)^2$ is a point such that $\Trop(X)=P=(\val(x_1),\val(x_2))\in\TT^2$. If $P$ is a vertex of $\Trop(Y)$, then $$\area(\Inf(P))=2\cdot\area(d(P))+\sum\limits_{\substack{V\in\I^\D(P),\\ V\ne P}} 1\cdot \area(d(V)),$$ which coincides with the definition of $\area(\Inf(P))$ in Definition~\ref{def_influence_vertex}. In Figure~\ref{governorship} (B) $\area(\Inf(P))$ is the area of the depicted part of the Newton polygon, with the area of the central face counted twice.
\end{example}

\begin{remark}
The dual object for a hypersurface is its Newton polytope. The dual
objects for the varieties of higher codimension are so-called {\it generalized Newton polytopes} or {\it
  valuations} in the McMullen
polytope algebra \cite{brion, erhart}. Even though for $X\subset Y\subset \KK^k$ with $\mathrm{codim}(Y)>1$ we can define $\I(\Trop(X))$ in a similar fashion (by intersecting the stars of cells of $\Trop(X)$ with $Y$),  it is not clear what is the right substitute for $\vol(\Inf(\Trop(X)))$ in this case.
\end{remark}

\subsection{General position of points with respect to the Newton polytope}

\begin{definition}
\label{def_generalposition}
A collection of tropical subvarieties $Z_1,Z_2,\dots,Z_n\subset \TT^k$ is {\it in general
position} with respect to a polytope $\D\subset \ZZ^k$ if for each $m=1,2,\dots k+1$ for each collection of indices $i_1<i_2<\dots<i_m$ the intersection $\Star^\D(Z_{i_1})\cap \Star^\D(Z_{i_2})\cap\dots\cap \Star^\D(Z_{i_{m}})$ (Definition~\ref{def_unionstar}) has codimension at least $m$ in $\TT^k$. In particular, the intersection of any $k+1$ stars is empty.
\end{definition}

\begin{prop}
\label{prop_points}
Let $v\in P(\ZZ^2)$ be a primitive vector such that  $v\notin P(\D)$ (Definition~\ref{def_star2}). Let $l$ be the line $\{t\cdot v|t\in\RR\}$. Then for any $n\in\ZZ_{>0}$ there exists a collection of points $P_1,P_2,\dots,P_n\in l\cap\ZZ^2$ in general position with respect to $\D$.
\end{prop}

\begin{proof}
Let $P_1=0$. Take any point $P_2\in l\cap\ZZ^2$. Note that the vector $P_1P_2\notin P(\D)$, therefore $P_1,P_2$ are in general position. Draw $\Star^\D(P_1),\Star^\D(P_2)$ and all stars $\Star^\D$ for the points of intersection of lines in $\Star^\D(P_1),\Star^\D(P_2)$. It is a finite collection of lines, none of them is $l$. Therefore we can find a point $P_3\in l\cap \ZZ^2$ which does not belong to these lines. Then we draw $\Star^\D(P_3)$ and starts $\Star^\D$ of all points of intersection between lines in $\Star^\D(P_3)$ and lines in $\Star^\D(P_1),\Star^\D(P_2)$. Again, we see a finite collection of lines and we can choose $P_4\in l\cap\ZZ^2$ not on these lines. We can continue to choose points $P_5,\dots, P_n$ as above. When we choose $P_n$ we have ${{n-1}\choose 2} |P(\D)|^2$ points of intersections of lines in $\Star^\D(P_1),\Star^\D(P_2),\dots,\Star^\D(P_{n-1})$, when we draw stars $\Star^\D$ for them, we have ${{n-1}\choose 2}|P(\D)|^3$ lines, therefore we can choose $P_n\in l\cap [0, R]^2$ where $R=|v|\cdot 2{{n-1}\choose 2} |P(\D)|^3$.
\end{proof}

Let $T_{v}$ be the translation $\TT^k\to \TT^k$ by the vector $v\in \RR^k$, i.e. $x\to x+v$.

\begin{prop}
\label{N}
 For a polytope $\D\subset \ZZ^k$ and a given set $Z_1,Z_2,\dots,Z_n\in \TT^k$ of tropical varieties there exists a set of vectors $v_1,v_2,\dots,v_n\in\ZZ^k$ such that the tropical varieties $T_{v_i}(Z_i)$ are in general position with $\D$.
 \end{prop}

\begin{proof}
Indeed, each star $\Star^\D(Z_i)$ is a finite union of hyperplanes (Definition~\ref{def_star2}). We argue as in Proposition~\ref{prop_points}. We can choose $v_1=0$ and $v_2\in\ZZ^k$ such that the intersection of each two hyperplanes $L_1,L_2$ from the collections $\Star^\D(Z_1)$ and $\Star^\D(T_{v_2}(Z_2))$ respectively is a linear subspace of dimension at most $k-2$. Then we choose a vector $v_3\in\ZZ^k$ such that the intersection of each pair of hyperplanes from different collections $\Star^\D(T_{v_i}(Z_i)), i=1,2,3$ is of dimension at most $k-2$ and the intersection of a triple of hyperplanes from different collections is of dimension at most $k-3$, etc. Each time we see that the set of $v_i$ such that $\Star^\D(T_{v_i}Z_i)$ violates some transversality condition is a finite union of hyperplanes, and we need to take a $v_i\in\ZZ^k$ outside of it.
\end{proof}


\begin{corollary}
For each $n,k\in \mathbb N, \Delta$ there exists a set of
points (taken as tropical varieties of dimension zero) $P_1,P_2\dots,$ $P_n\in \mathbb Z^k\subset\TT^k$ in general
position with respect to $\Delta$. 
\end{corollary}

\begin{corollary}
\label{cor_volume}
For a collection of tropical varieties $Z_1,Z_2,$ $\dots,Z_n\subset Y\subset \TT^k$ in general position with respect to the Newton polytope $\Delta$ of a tropical hypersurface $Y$, the
sum $\sum_{i=1}^n \vol(\Inf(Z_i))$ is at most $k\cdot\mathrm{Volume}(\Delta)$.
\end{corollary}

\begin{proof} This follows from the definitions of a general position (Definition~\ref{def_generalposition}) and multiplicities in the volume of $\Inf$ (Definition~\ref{def_influencegeneral}). \end{proof}

\subsection{First applications}
Consider an algebraci curve $C\subset (\KK^*)^2$ passing through
$p_1,p_2,\dots,p_n\in (\KK^*)^2$ with multiplicities
$m_1,m_2,\dots,m_n$ respectively. Suppose that $n\geq 2$ and the minimal lattice width $\omega(\D)$ of the Newton polygon $\D$ of $C$ satisfies $\omega(D)\geq\max(m_i)$.

\begin{lemma}\label{lemma_old}
If the points $\Val(p_i)\in\ZZ^2, i=1,\dots,n$ are in general
position with respect to $\Delta$ (see Proposition \ref{N} and its
corollaries), then the area of $\Delta$ satisfies the inequality
\begin{equation}
\mathrm{area}(\Delta)\geq\frac{1}{4}\sum_{i=1}^n m_i^2.
\end{equation}
\end{lemma}

\begin{proof}  Theorem \ref{infl_theorem} and Corollary \ref{cor_volume} imply that
$$\sum_{i=1}^n\frac{m_i^2}{2}\leq\sum\limits_{i=1}^n \area(\Inf(P_i))\leq 2\cdot\area(\D).$$
\end{proof}

\begin{corollary}
\label{cor_first}
Suppose that a curve of degree $d$ passes through the points $p_1,p_2,\dots,p_n\in (\KK^*)^2$ with multiplicities
$m_1,m_2,\dots,m_n$ respectively, $d\geq\max(m_i)$, and $n\geq 2$. Suppose also that the points $\Val(p_i)$ are in general position with respect to $\D$.  Then, we
have $d^2\geq\frac{1}{2}\sum_{i=1}^n m_i^2$.
\end{corollary}

\begin{proof}
The equation of a curve of degree $d$ may contain some monomials with zero coefficients. So, if the minimal lattice width of the actual Newton polygon of $C$ is at least $\max(m_i)$, then we conclude by Lemma~\ref{lemma_old}. If it is not the case, we apply Lemma~\ref{lemma_width}.

If $C$ has a rational component passing through a point $p=(x_1,y_1)$ parametrized by $s$ as $(x_1s^{k},y_1s^{l})$, then $C$ is reducible, and we can
perturb this component, because it does not pass through other $p_i$ by genericity (recall that the tropicalization of this component is a straight line in the direction $(k,l)\in P(\D)$). After that this component is no longer of the type
$(x_1s^k,y_1s^l)$, and this perturbation does not change the degree of the curve. After repeating this cycle of arguments necessary number of times we can apply Lemma~\ref{lemma_old}.
\end{proof}

\begin{remark}
\label{rem_surface} Consider a hypersurface $H$ in $(\KK^*)^3$ passing through generic points of multiplicity two. The classification of possible combinatorial neighborhoods of a two-fold point $P$ in a tropical surface in $\TT^3$ (\cite{markwig2}) allows us to produce an estimates for the volume and the shape of $\Inf(P)$. 

We can prove that $\vol(\Inf(P))\geq 2$ if $P=(\val(p_1),\val(p_2))$ for a point $(p_1,p_2)$ of multiplicity two in $H$. With a few more work (one should check possible intersections $\Inf(P)\cap \Inf(Q)$ for different points $P,Q$ in general position) the author obtained a proof of an estimate $n\leq \frac{d^3}{3}$ for the degree $d$ of a surface with $n$ two-fold points, as we did in Lemma~\ref{lemma_old}.

However, the theorem of Alexander and Hirschowitz provides a better estimate $n\leq \frac{(d+1)(d+2)(d+3)}{24}$. Nevertheless, we expect that this research paradigm can be carried out for points of multiplicity $m$ on hypersurfaces in $(\KK^*)^3$, with a conjectural estimate $\vol(\Inf(P))\geq cm^3$ for some constant $c$. 
\end{remark}


\begin{remark}
We expect that for a line $L$ of multiplicity $m$ inside a surface $Y$ of degree $d$ in $\KK P^3$ the estimate $\vol(\Inf(\Trop(L)))\geq cm^2d$ holds with some constant $c$. This will give an estimate for the degree of  a surface with multiple two-fold points and $m$-fold lines. The idea is how it could work is as follows. Consider $\Trop(L)\subset \Trop(Y)$. Intersect it with a plane $\{Z=\mathrm{const}\}\subset\TT^3$. In the intersection we will see the same picture as for planar curves with an $m$-fold point. Varying the constant and the normal vector of such a plane we can make a conjecture as follows.
\end{remark}

\begin{conj}
Under the above hypothesis, $\Inf(\Trop(L))$ is a a connected subset of the Newton polytope $\D$ of $Y$, which intersects the faces of $\D$, perpendicular to the directions of the rays of $\Trop(L)$, and whose sections by planes with a primitive normal vector in $\ZZ^3$ have area at least $\frac{3}{8}m^2$. Hence, $\vol(\Inf(\Trop(L)))\geq \frac{3}{8}m^2d$.
\end{conj}

\subsection{Detropicalization Lemma}

An algebraic statement over an algebraically closed field sometimes implies the
same statement over all fields of the same
characteristic. Tropical geometry may help in such a situation, see \cite{tyomkin}. Another application of tropical geometry in number theory is \cite{katz2015uniform}. This
section describes a particular application of this principle to our estimate.

Recall that our field $\KK$ is the field of fractions $\frac{f(t)}{g(t)}$ where $f,g\in\FF[t]$. Note that we can substitute $t=a$ if $g(a)\ne 0$.

Let us recall how to {\it tropicalize} the problem of curves' counting. We
would like to count plane complex algebraic curves of given genus and
degree, these curves are required to pass through a number of generic
points $q_1,q_2,\dots,q_l\in \CC P^2$ ($l$ is chosen in such a way
that the number of curves is expected to be finite). Since the points are
generic, we can force them to go to infinity with some asymptotics, say
$q_i=(t^{x_i},t^{y_i}), (x_i,y_i)\in\ZZ^2$. Then we consider the limits of these
curves $C_t$ under the function $\log_t(|z|):\CC^2\to \RR^2$. This is
more or less the same as if we consider a curve $C$ over $\KK$
passing through $(t^{x_i},t^{y_i})\in (\KK^*)^2$ and then
take its tropicalization $\Trop(C)$. Hence we started from $\CC$, lifted to
$\KK$, and finally descended to $\TT$.

Detropicalization is the opposite process: we prove something
in $\TT$, then lift the construction to
$\KK$, and return to $\FF$ using such a substitution for an appropriate $a$. We establish the following lemma.

\begin{lemma}
\label{lemma_trop}
Let $m_1,m_2,\dots, m_n$ be non-negative integers. Let $\D$ be a lattice polygon such that $$\area(\D)<\sum\limits_{i=1}^n \frac{m_i^2}{4}.$$ Then, if a set of points $P_1,\dots,P_n\in \ZZ^2\subset\TT^2$ is in general position with respect to $\D$ (Definition~\ref{def_generalposition}), then for each valuation field $\KK$ and points $p_1,p_2,\dots,p_n\in(\KK^*)^2$ such that $\Val(p_i)=P_i$ there exists no curve $C$ over $\KK$ with the Newton polygon $\D$, with $\mu_{p_i}(C)\geq m_i, i=1,\dots, n$.
\end{lemma}
\begin{proof}
Suppose that such a curve $C$ exists. Then, consider $\Trop(C)$. We know that in this case $$\area\big(\Inf(P_i)\big)\geq \frac{m_i^2}{2}$$ for $i=1,\dots,n$ and, therefore, $\sum_{i=1}^n\area(\Inf(P_i))\geq \sum\limits_{i=1}^n \frac{m_i^2}{4} \geq 2\cdot\area(\D)$. So, using Corollary~\ref{cor_volume} we arrive at a contradiction.
\end{proof}

\begin{lemma}[Detropicalization lemma]
\label{lemma_detrop}
Let $\KK$ be the field of fractions of $\FF[t]$. Suppose that there exists no curve $C$ over $\KK$ with the Newton polygon $\D$ such that $$\mu_{(t^{-x_i},t^{-y_i})}(C)\geq m_i,$$
for given different points $(x_i,y_i)\in\ZZ^2, i=1,\dots,n$ and given numbers $m_i\in\ZZ_{>0},i=1,\dots,n$.  Then, there exists a constant $N$ depending on $m_1,m_2,\dots,m_n,\D, \max |x_i|,\max |y_i|$ with the following property.  If $|\FF|\geq N$, then there exists $a\in \FF$ such that
there is no curve $C$ over $\FF$ with the Newton polygon $\D$, satisfying $\mu_{ {(a^{-x_i},a^{-y_i})}}(C)\geq m_i$ for each $i=1,\dots,n.$
\end{lemma}

\begin{proof}  Suppose the contrary. Take any $b\in \FF$.  All the constraints imposed by the fact $\mu_p(C)\geq m$ are linear
equations in the coefficients of the equation of $C$. Therefore the only reason why there is no solution for this system over
$\KK$ and there exists a solution over $\FF$ is that some
minor of the matrix of the equations turns out to be 0 after substitution $t=b$. Thus, let us compute all
the considered minors before, they reveal to be polynomials in $t$ with degrees
depending on our data. Therefore, $b$ is a root of this fixed polynomial of a certain bounded degree. Obviously, if $|\FF|$ is big enough, then there exists
$a$ which is not a root of this polynomial. Therefore, this $a$ satisfies the statement of lemma.
\end{proof}

\begin{remark}
In a similar way we can ``detropicalize'' in other situations, if
the conditions imposed on $C$ reveal to be algebraic conditions on the
coefficients of the equation of $C$.
\end{remark}

\section{Degeneration of tropical points to a line}
In this section, using tropical floor diagrams (see \cite{BIMS,floor1}), we construct a special collection of points in $\TT^2$,
which are in general position with respect to the Newton polygon $\D$; this
construction gives another estimate for $\area(\D)$.

Consider a tropical curve $H$ given by  $\Trop(F)(X,Y)=\max_{(i,j)}(A_{ij}+iX+jY)$ where $(i,j)$ runs over lattice points in a fixed Newton polygon $\D$. After a toric change of coordinates we may assume that the minimal lattice width $\omega(\D)$ of $\D$ is attained in the horizontal direction.
Let $\Delta$ be contained in the strip $\{(x,y)|0\leq y\leq N\}$.
Let us choose points $P_1,P_2,\dots,P_n$ on the line $l=\{(X,Y)|Y=\frac{1}{N+1} X\}$ which is almost
horizontal, namely, its slope $\frac{1}{N+1}$ is less than any possible slope
of a non-horizontal edge of a tropical curve with the given Newton polygon $\Delta$. 

\begin{prop}
Suppose that each of the points $P_1,P_2,\dots,P_n$ is not a vertex of $H$, and each $P_i$ belongs to a horizontal edge $E_i$ of $H$. In this case, for each $1\leq i<j\leq n$ we have $\Inf(P_i)\cap\Inf(P_j)=\varnothing$. 
\end{prop}

\begin{proof}
Indeed, in this case the vertices in $\I(P_i)$ belong to the horizontal line through $P_i$ for each $i=1,2,\dots,n$, and all $P_i$ have different $y$-coordinates.
\end{proof}

\begin{corollary}
In the above case, $\sum_{i=1}^n\area(\Inf(P_i))\leq \area(\D)$.
\end{corollary}

In general, a correction term (Eq.~\ref{eq_correction}) appears by the following reasons. The line $l$ is subdivided by intersections with $H$, each connected component of $l\setminus (l\cap H)$ corresponds  to a monomial in $\Trop(F)$, i.e. to a lattice point in $\D$. 

Moving along $l$ from left to right and marking corresponding lattice points in $\D$ we obtain a lattice path in $\D$, which possesses the following property:
each edge in this path is either vertical (and has positive projection on the vertical line), or has positive projection on
the horizontal line. Indeed, if $A_{ij}+iX+jY>A_{i'j'}+i'X+j'Y$, but for small $\varepsilon>0$  $$A_{ij}+i(X+\varepsilon)+j(Y+\frac{1}{N+1}\varepsilon)<A_{i'j'}+i'(X+\varepsilon)+j'(Y+\frac{1}{N+1}\varepsilon),$$ then the vector $(i'-i,j'-j)$ has the described above properties.

If $P_i$ is not a vertex of $H$, and $P_i$ belongs to an edge
$E_i$ of $H$, then denote by $s_i$ the length of the
horizontal projection of $d(E_i)$.
If $P_i$ is a vertex of $H$, then denote by $s_i$ the
length of the horizontal projection of $d(P_i)$.

Above considerations show that  $\sum\limits_{i=1}^ns_i\leq \omega(\D)$, see Figure~\ref{fig:main} for illustration.

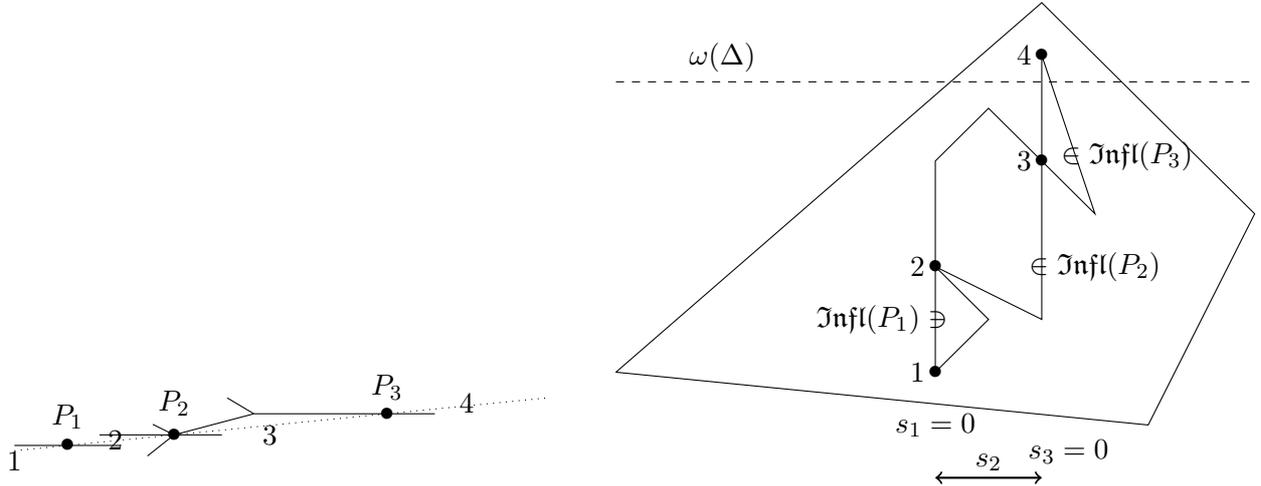
\begin{figure} [htbp]
\begin{center}
\begin{tikzpicture}[scale=0.7]
\draw[dotted] (0,0)--(10,1);
\draw (1,0.1) node {$\bullet$};
\draw (1,0.65) node {$P_1$};
\draw (0,-0.2) node {$1$};
\draw (3,0.3) node {$\bullet$};
\draw (3,0.9) node {$P_2$};

\draw (1.9,0.2) node {$2$};
\draw (7,0.7) node {$\bullet$};
\draw (7,1.2) node {$P_3$};
\draw (4.8,0.3) node {$3$};
\draw (8.5,0.9) node {$4$};

\draw (0,0.1)--(2,0.1);
\draw (1.6,0.3)--(3.9,0.3);
\draw (3,0.3)--(2.5,-0.1);
\draw (3,0.3)--(2.6,0.5);
\draw (3,0.3)--(4.5,0.7);
\draw (4.5,0.7)--(4,1);
\draw (4.5,0.7)--(7.9,0.7);
\end{tikzpicture}
\qquad
\begin{tikzpicture}[scale=0.7]
\draw (0,0) --(10,-1)--(12,3)--(8,7)--(0,0);
\draw (6,0)--(6,2)--(7,1)--(6,0);
\draw (6,2)--(6,4)--(7,5)--(8,4)--(8,1)--(6,2);
\draw (8,4)--(8,6)--(9,3)--(8,4);

\draw[<->, thick] (6,-2)--(8,-2);
\draw (7,-1.7) node {$s_2$};
\draw (8.5,-1.5) node {$s_3=0$};

\draw (6,-1) node {$s_1=0$};
\draw (5,1) node {$\Inf(P_1){}\ni$};
\draw (9,2) node {$\in{}\Inf(P_2)$};
\draw (9.6,4.1) node {$\in{}\Inf(P_3)$};
\draw (6,0) node [below, left] {$1$};
\draw (6,2) node [below, left] {$2$};
\draw (8,4) node [below, left] {$3$};
\draw (8,6) node [below, left] {$4$};
\draw (6,0) node {$\bullet$};
\draw (6,2) node {$\bullet$};
\draw (8,4) node {$\bullet$};
\draw (8,6) node {$\bullet$};

\draw[thin, dashed] (0,5.5)--(12,5.5);
\draw (2,5.5) node[above]{$\omega(\D)$};

\end{tikzpicture}
\caption {The left picture represents a part of a tropical curve through points $P_1,P_2,P_3$ on an almost horizontal line. The second picture is dual to the first  picture, we see parts of the regions of influence of the points $P_1,P_2,P_3$. The marked points $1,2,3,4$ represent the monomials which are maximal on the parts of the dotted line on the left picture. The lattice path $1,2,3,4$ is non-decreasing by the $x$-coordinate, therefore $\sum_{i=1}^n s_i\leq \omega_{(1,0)}(\D)$. The vertical projections of the intervals $12,23,34$ in the right picture are $s_1=0,s_2,s_3=0$ respectively.}
\label{fig:main}
\end{center}
\end{figure}

\begin{prop}
\label{prop_main}
In the above notation, and using Definition~\ref{def_I} for $\I_{P_i}((1,0))$ we have $$\frac{1}{2}\sum\limits_{i=1}^{n} (m_i^2-s_i^2)\leq \sum\limits_{i=1}^{n}\left(\sum\limits_{V\in \I_{P_i}((1,0))}\area(d(V))\right)\leq \area(\D).$$
\end{prop}

\begin{proof}
The right inequality is trivial, because the sets $\I_{P_i}((1,0))$ do not intersect each other. The left inequality follows from the estimate $$\sum\limits_{V\in \I_{P_i}((1,0))}\area(d(V))\geq \frac{1}{2}(m_i^2-s_i^2)$$ for each $i=1,\dots,n$. Indeed, if $P_i$ is not a vertex of $H$, then $P_i$ belongs to an edge $E_i$ of $H$. If $E_i$ is horizontal (the case for $i=1,3$ in Figure~\ref{fig:main}), then $d(E_i)$ is vertical, and $s_i=0$. By Lemma~\ref{lemma_prep} (or Theorem~\ref{infl_theorem} for the case when $P_i$ belongs to an edge) $$\sum\limits_{V\in \I_{P_i}((1,0))}\area(d(V))\geq \frac{1}{2}m_i^2.$$  If $E_i$ is not horizontal, then $s_i\geq m_i$ (because $s_i$ is at least the lattice length of $d(E_i)$ and we use Remark~\ref{rem_length} in this case) and our inequality becomes trivial. If $P_i$ is a vertex of $H$, then the inequality follows from Lemma~\ref{lemma_prep}, because in this case $$\sum\limits_{V\in \I_{P_i}((1,0))}\area(d(V))\geq (m_i-s_i)\cdot s_i+ \frac{1}{2}(m_i-s_i)^2 = \frac{m_i^2-s_i^2}{2}.$$

Indeed, the term $\frac{1}{2}(m_i-s_i)^2$ comes from Lemma~\ref{lemma_prep}, and $(m_i-s_i)\cdot s_i$ estimates the area of $d(P_i)$ from below, because $\omega_{(1,0)}(d(P_i))=s_i$ by our assumption and $d(P_i)$ has two vertical sides of length at least $m_i-s_i$ by Lemma~\ref{lemma_prep}.
\end{proof}

\begin{proof}[Proof of Theorem~\ref{main_theorem}]
By Proposition~\ref{prop_points} there exists $N$ such that
there exists a generic with respect to $\D$ collection of lattice points on the line $y=\frac{1}{\omega(\D)+1}x$ with $|x_i|,|y_i|<N$. Then,  Proposition~\ref{prop_main} implies the statement over $\KK$ (as in Lemma~\ref{lemma_trop}), and Lemma \ref{lemma_detrop}, ``detropicalizing'', concludes the proof. 
\end{proof}


\section{Speculations destined to coding theory}
\label{codetheory}

In informatics, (error-correcting) {\it coding-theory} deals with subsets
$C\subset A^n$ ($A$ is a finite set) which are as big as possible, and the Hamming distance $d$
between the elements of $C$ is also as big as possible, i.e. we maximize
$\delta = \min_{a,b\in C, a\ne b} d(a,b)$. Such a subset $C$ is called a
{\it code} and it is suitable for the following problem. We transmit
a message which is an element of $C$. If, during the transmission
procedure, the message does change in at most $\frac{\delta-1}{2}$
positions, then we can uniquely repair it back, that is why this is called
an {\it error-correcting} code. As an introductory book, which relates
this subject to algebraic geometry, see
\cite{tsfasman}.  Studying of singular varieties is related with code-theory (\cite{wahl1998nodes}), for the relation of this topic with Seshadri constants (which is a relative of Nagata's conjecture), see \cite{hansen2001error}.  

Finding such subsets $C$ is a hard combinatorial problem. A particular
source for codes is the set of linear subspaces of $\FF_q^n$ ({\it
  linear codes}), mostly
because they have a comparatively simple description. A common construction is
the following. We choose a set of points $p=\{p_1,p_2,\dots,p_n\}\subset \FF_q^m$ and
consider a set $V_d\subset \FF_q[x_1,x_2,\dots,x_m]$ of polynomials
of degree at most $d$ (one can consider any linear system on
a toric variety and points on it as well).
Then we take the evaluation map: $$ev_p:V_d\to \FF_q^n,
ev_p(f)=(f(p_1),f(p_2),\dots,f(p_n)).$$ The image of $ev_p$ is a linear
code which is quite simple to calculate, but the problem is how to
choose the points $p_i$ such that there is no polynomial which vanishes at
the chosen points (otherwise we have to deal with the kernel of $ev_p$) and how to estimate the minimal distance $\delta$. For example, one may take all the points with all non-zero coordinates as $p$, $p=\FF_q^m$. However, since we are trying to minimize the codimension of the code in $\FF_q^n$ this choice is far from optimal.
   
Following a suggestion of Joaquim Ro\'e, we mention here a way we can
exploit the main ideas of this article to construct a linear code,
which uses not too many points and provides a map, similar to $ev_p$,
without a kernel.

In the previous sections, for a given polygon $\D$ and numbers
$m_1,m_2,\dots, m_n$ we constructed the set of points
$p=\{p_1,p_2,\dots,p_n\}\subset (\FF_q^*)^2$ such that there is no curve $C$ with Newton polygon $\Delta$, possessing the property $\mu_{p_i}(C)\geq
m_i$ for each $i$. Recall that for this construction we should
carefully choose points $(x_i,y_i)\in\ZZ^2, i=1,\dots, n$, then, for $q$ big enough
there exists an $a\in\FF_q$, such that the points $p_i=(a^{-x_i},a^{-y_i})$
possess the required properties.
 
\begin{example}
Consider $\Delta=[0,1,\dots,d]\times [0,1\dots,N]\subset\ZZ^2$. It follows from the proof of Theorem~\ref{main_theorem} that if we
put $n$ points $p_1,p_2,\dots,p_n$ of the same multiplicity $m\leq\min(N,d)$ along the curve $y^{\omega(\D)+1}=x$,  then there is no algebraic curve $C$ with the Newton polygon $\D$ and $\mu_{p_i}(C)\geq m$ if the inequality $dN< \frac{1}{2}(n-d/m)m^2$ holds.
\end{example}
Therefore, taking $N< \frac{(n-d/m)m^2}{2d}$ we construct the
evaluation map $ev:\FF_q^{dN}\to \FF_q^{\frac{nm(m+1)}{2}}$ with a
trivial kernel. For this map, we take a point $f\in\FF_q^{dN}$, which we treat as a polynomial $F$ with the Newton
polygon $\Delta$, then take the coefficients of $\left(F\ \mathrm{mod}\ {I_{p_i}^m}\right)$ for
each $i=1,\dots,n$. This bunch of numbers gives the image $ev(f)$. In this construction we immediately see that the minimal non-zero Hamming distance $\delta$ is at most $n$, because the image of $f\equiv 1$ under the map $ev$ contains exactly $n$ non-zero elements.  
\begin{conj}
This estimate is sharp, i.e. $\delta=n$ for this code.
\end{conj}

\bibliography{./nagata}
\bibliographystyle{abbrv}

\end{document}